\documentclass[10pt]{article}

\usepackage[english]{babel}
\usepackage[utf8]{inputenc}

\usepackage{amsmath,amssymb,amsthm}
\usepackage[linktocpage,unicode]{hyperref}

\usepackage{tikz,ytableau,rotating}
\usepackage{fullpage}

\hypersetup{
    pdftoolbar=true,        
    pdfmenubar=true,        
    pdffitwindow=false,     
    pdfstartview={FitH},    
    pdftitle={A q-analog of Schläfli and Gould identities on Stirling numbers},
    pdfauthor={Matthieu Josuat-Vergès},
    pdfsubject={combinatorial identities},
    pdfkeywords={Stirling numbers, generating functions, continued fractions, combinatorial identities, q-analog},
    pdfnewwindow=true,      
    colorlinks=true,       
    linkcolor=red,          
    citecolor=blue,        
    filecolor=magenta,      
    urlcolor=cyan,           
    unicode=true          
}

\newtheorem{theo}{Theorem}[section]
\newtheorem{coro}[theo]{Corollary}
\newtheorem{lemm}[theo]{Lemma}

\theoremstyle{definition}
\newtheorem{defi}[theo]{Definition}
\newtheorem{rema}[theo]{Remark}

\newcommand{\qbin}[2]{\genfrac{[}{]}{0pt}{}{#1}{#2}_q}

\DeclareMathOperator{\cro}{cr}
\DeclareMathOperator{\inv}{inv}
\DeclareMathOperator{\rlm}{rlm}

\title{A $q$-analog of Schläfli and Gould identities on Stirling numbers}

\author{ Matthieu Josuat-Vergès  \thanks{Supported by ANR CARMA (ANR-12-BS01-0017). }  \\
 CNRS and Institut Gaspard Monge \\
 Université Paris-Est Marne-la-Vallée \\
 France \\
 \small \tt matthieu.josuat-verges@u-pem.fr
}

\date{}

\begin{document}
      
\maketitle

\begin{abstract}
Stirling numbers of both kinds are linked to each other via two combinatorial identities due
to Schläfli and Gould. Using $q$-analogs of Stirling numbers defined as inversion generating
functions, we provide $q$-analogs of the two identities. The proof is computational and
we leave open the problem of finding a more combinatorial one.
\end{abstract}

\section{Introduction}

(Unsigned) Stirling numbers of the first and second kinds, respectively denoted
$S_1(n,k)$ and $S_2(n,k)$, form two infinite triangular arrays of integers with numerous combinatorial
properties (see
\href{https://oeis.org/A130534}{A130534} and \href{https://oeis.org/A008277}{A008277} in \cite{sloane}). 
They can be defined via their exponential generating functions:
\begin{align*}
 \sum_{0 \leq k \leq n }  S_1(n,k) x^k \frac{z^n}{n!} = \frac{1}{(1-z)^x}, \qquad
 \sum_{0 \leq k \leq n }  S_2(n,k) x^k \frac{z^n}{n!} = e^{ x (e^z-1) }.
\end{align*}
A well-known way to relate the two Stirling triangles is to define two matrices $A$ and $B$ by 
\begin{align*}
   A_{i,j} = \begin{cases}
                 (-1)^{i-j} S_1(i,j)  & \text{if } i\geq j, \\
                 0                    & \text{otherwise},
             \end{cases} \qquad
   B_{i,j} = \begin{cases}
                 (-1)^{i-j} S_2(i,j)  & \text{if } i\geq j, \\
                 0                    & \text{otherwise},
             \end{cases}
\end{align*}
which are then related by $A^{-1} = B$.
But there is another way: the following two identities express Stirling numbers of the first and second kind in terms of each other,
in a symmetric way:
\begin{align} \label{id1}
   S_1(n,k) &= (-1)^{n-k} \sum_{j=0}^{n-k} (-1)^j
   \binom{n-1+j}{n-k+j} \binom{2n-k}{n-k-j} S_2(n-k+j,j), \\
 \label{id2}
   S_2(n,k) &= (-1)^{n-k} \sum_{j=0}^{n-k} (-1)^j
   \binom{n-1+j}{n-k+j} \binom{2n-k}{n-k-j} S_1(n-k+j,j).
\end{align}
The first one is essentially due to Schläfli~\cite{schlafli1852,schlafli1867}, and to Schlömilch~\cite{schlomilch} 
in a different form. It is particularly interesting because we can deduce an exact formula for $S_1(n,k)$ as a double 
sum, using the exact formula 
\begin{equation*}
 S_2(n,k) = \frac{1}{k!} \sum_{j=0}^k (-1)^{k-j} \binom{k}{j} j^n.
\end{equation*}
One century later, Gould \cite{gould} obtained the second one.
The proofs of both results rely on the analysis of exponential generating functions of Stirling numbers,
and are present in several textbooks such as \cite{charalambides}, \cite{comtet}, \cite{quaintance}.
See also \cite{gouldkwongquaintance} and \cite{sun} for recent works related with these identities.

Our goal is to show that both \eqref{id1} and \eqref{id2} remains true when 
$S_1(n,k)$ and $S_2(n,k)$ are replaced with natural $q$-analogs, denoted $S_1[n,k]$ and $S_2[n,k]$
(see Theorems~\ref{thqid1} and~\ref{thqid2}).
The $q$-analog $S_2[n,k]$ was studied in \cite{rubeyjosuatverges}, and
the $q$-analog $S_1[n,k]$ is new and gives a natural companion to $S_2[n,k]$.
Indeed, they share similar properties: they can be defined combinatorially in terms of inversions
(see Definitions~\ref{defqst2} and~\ref{defqst1}),
and their ordinary generating functions have continued fraction expansions
(see Theorems~\ref{qfrac2} and~\ref{qfrac1}), from which 
we can obtain closed formulas in terms of binomial and $q$-binomial coefficients
(see Theorems~\ref{thformulaqst2} and~\ref{thformulaqst1}).

Note that this study of the new $q$-analog $S_1[n,k]$ is not just for its own sake,
since the proofs of Theorems~\ref{thqid1} and~\ref{thqid2} rely on the formula in Theorem~\ref{thformulaqst1}.
Once we know the closed formulas for $S_1[n,k]$ and $S_2[n,k]$, which take the form of expansions of 
$(1-q)^{n-k}S_1[n,k]$ and $(1-q)^{n-k}S_2[n,k]$ as linear combinations of the 
``shifted'' $q$-binomial coefficients 
\[ 
  q^{\binom {j+1} 2 } \qbin i j, 
\]
the proof of our identities are essentially a coefficient wise calculation. In particular, they do
not involve $q$-binomials, they are binomial identities of hypergeometric type.

The existence of such $q$-analogs for the two identities is somewhat unexpected, for two different reasons. 
Firstly, they are extremely simple, since each Stirling number is just replaced with its $q$-analog
(this simplicity is the reason why we were able to find the identities via computer experiments).
Secondly, the known proofs of the original identities rely on the \emph{exponential} generating functions of
Stirling numbers, whereas the $q$-Stirling numbers only have nice \emph{ordinary} generating functions. 
The proof is not really enlightening, since it is computational, and a more combinatorial proof is still
to be found. For example, it would be natural to try to interpret the alternating sums as an 
inclusion-exclusion process.

\section{Common definitions and formulas}

We use the standard notation for $q$-binomial coefficients: if $0\leq k \leq n$, 
\begin{align*}
 \qbin{n}{k} = \frac{[n]_q!}{[k]_q![n-k]_q!},
\end{align*}
where $[n]_q! = [1]_q [2]_q \dots [n]_q$ and $[i]_q = 1+q+\dots + q^{i-1}$. And $\qbin{n}{k}=0$ if $k<0$ or $k>n$. We have:
\begin{align} \label{recqbin}
 \qbin{n}{k} = \qbin{n-1}{k-1} + q^k \qbin{n-1}{k} = q^{n-k} \qbin{n-1}{k-1} + \qbin{n-1}{k}.
\end{align}

The Pochhammer symbol is $(\alpha)_m = \alpha(\alpha+1)\dots(\alpha+m-1)$.
The general hypergeometric series is:
\begin{equation*}
    {}_{r+1} F_r \Big( \begin{array}{c} \alpha_1 \, ; \, \dots \, ; \, \alpha_{r+1} \\ \beta_1 \, ; \, \dots \, ; \, \beta_r \end{array}  \Big| \, z \,  \Big)
   = \sum_{n\geq0} \frac{ (\alpha_1)_n \dots (\alpha_{r+1})_n z^n}{ (\beta_1)_n \dots (\beta_r)_n n!}.
\end{equation*}
We will need the Pfaff-Saalschütz summation theorem (see \cite[Section~4.4]{aigner} or \cite[Chapter 5, \S 51]{rainville}):
\begin{equation} \label{pfaffsaalschutz}
   {}_3 F_2 \Big( \begin{array}{c} -m \, ; \, \alpha \, ; \, \beta \\ \gamma \, ; \, \alpha+\beta-\gamma-m+1 \end{array}  \Big| \, 1 \,  \Big)
   = \frac{ (\gamma-\alpha)_m (\gamma-\beta)_m }{ (\gamma)_m (\gamma-\alpha-\beta)_m },
\end{equation}
where $m$ is a nonnegative integer. Note that we have in particular, by letting $\beta$ tend to infinity:
\begin{equation} \label{gauss}
   {}_2 F_1 \Big( \begin{array}{c} -m \, ; \, \alpha \\ \gamma \end{array}  \Big| \, 1 \,  \Big)
   = \frac{ (\gamma-\alpha)_m }{ (\gamma)_m }.
\end{equation}
These equalities hold as long as the denominator in the right hand side does not vanish.

We also need the {\it contiguity relation}:
\begin{align} \label{eqcontig}
\begin{split}
  (\alpha_3-\alpha_4) \times\, {}_4 F_3 \Big( \begin{array}{c} \alpha_1 \, ; \,  \alpha_2 \, ; \, \alpha_3 \, ; \, \alpha_4 \\ \beta_1 \, ; \, \beta_2 \, ; \, \beta_3 \end{array}  \Big| 1 \Big)
  =
     \alpha_3 \times\, & {}_4 F_3 \Big( \begin{array}{c} \alpha_1 \, ; \,  \alpha_2 \, ; \, \alpha_3+1 \, ; \, \alpha_4 \\ \beta_1 \, ; \, \beta_2 \, ; \, \beta_3 \end{array}  \Big| 1 \Big) \\
   &- 
     \alpha_4 \times\, {}_4 F_3 \Big( \begin{array}{c} \alpha_1 \, ; \,  \alpha_2 \, ; \, \alpha_3 \, ; \, \alpha_4+1 \\ \beta_1 \, ; \, \beta_2 \, ; \, \beta_3 \end{array}  \Big| 1 \Big).
\end{split}
\end{align}
It is written here for a $ {}_4 F_3$ series but holds in general for ${}_{r+1} F_r $.
It is elementary to prove, but see \cite[Equation~14, p. 82]{rainville}.

We will use three kinds of lattice paths in the discrete quarter plane $\mathbb{N}^2$:
\begin{itemize}
 \item A {\it Dyck path} of length $2n$ is a path  from $(0,0)$ to $(2n,0)$ with steps
$(1,1)$ or $(1,-1)$. 
 \item A {\it Motzkin path} of length $n$ is a path from $(0,0)$ to $(n,0)$ with steps
$(1,1)$ or $(1,0)$ or $(1,-1)$. 
 \item A {\it Schröder path} of length $2n$ is a path from $(0,0)$ to $(2n,0)$ with steps
$(1,1)$ or $(1,-1)$ or $(2,0)$. 
\end{itemize}
Note that in a Dyck or Motzkin path (but not in a Schröder path), the number of steps equals the length. 
The $y$-coordinate of $(x,y) \in \mathbb{N}^2$ is called its {\it height}. Similarly, the height of 
a level step $(1,0)$ or $(2,0)$ in some particular path is well defined, and other steps in the path 
have an initial height and final height (the path is seen as going from left to right).

Suppose know that we assign {\it weights} to some steps of the path according to some specific rules. 
The weight is usually a number or a monomial in some indeterminates or a sum of those.
The weight of the path is the product of the weights of his steps, and the weighted generating function is the sum of such weights, over
some set of paths that depends on the context.
Also, if some step in a path has a weight $x+y$ we can as well consider that this path comes in two different ``colored'' versions,
one where this step has weight $x$ and the other where it is $y$. 

Weighted Motzkin paths are the appropriate tool for the combinatorial study of continued fractions that
naturally arise as the moment generating function of orthogonal polynomials (see \cite{aigner,flajolet}),
called J-fractions.

\begin{theo}[cf.~\cite{flajolet}]  \label{chemins}
Let $M_n$ be the weighted generating function of Motzkin paths of length $n$, where the weights are:
\begin{itemize}
 \item $a_i$ for a step $(1,-1)$ starting at height $i$, 
 \item $b_i$ for a step $(1,0)$ at height $i$.
\end{itemize}
Then we have:
\begin{equation*}
\sum_{ n \geq 0} M_n z^n =
\cfrac{1}{
    1 - b_0 z - \cfrac{a_1 z^2}{
        1 - b_1 z - \cfrac{ a_2 z^2}{
            1 - b_2 z - \cfrac{ a_3 z^2}{\ddots}
        }
    }
}.
\end{equation*}
\end{theo}

The theorem below is about another kind of continued fractions, called T-fractions.
They were shown by Roblet and Viennot \cite{robletviennot} to be connected with two-points Padé approximants 
and with an enumeration of Dyck paths with certain weights. 

\begin{theo}[cf.~\cite{robletviennot}]  \label{chemins2}
Let $D_{2n}$ be the weighted generating function of Dyck paths of length $2n$, where the weights are:
\begin{itemize}
 \item $b_i$ for a step $(1,-1)$ starting at height $i$ and following a step $(1,1)$,
 \item $c_i$ for a step $(1,-1)$ starting at height $i$ and following a step $(1,-1)$.
\end{itemize}
Then we have:
\begin{equation*}
\sum_{ n \geq 0} D_{2n} z^n =
\cfrac{1}{
    1 - (b_1-c_1) z - \cfrac{ c_1 z}{
        1 - (b_2-c_2) z - \cfrac{ c_2 z}{
            1 - (b_3-c_3) z - \cfrac{ c_3 z}{\ddots}
        }
    }
}.
\end{equation*}
\end{theo}

Also, note that the straightforward analog of Theorem~\ref{chemins} in the case of T-fractions gives the 
following:

\begin{theo}  \label{chemins3}
Let $S_{2n}$ be the weighted generating function of Schröder paths of length $2n$, where the weights are:
\begin{itemize}
 \item $c_i$ for a step $(1,-1)$ starting at height $i$, 
 \item $d_i$ for a step $(2,0)$ at height $i$.
\end{itemize}
Then we have:
\begin{equation*}
\sum_{ n \geq 0} S_{2n} z^n =
\cfrac{1}{
    1 - d_0 z - \cfrac{ c_1 z}{
        1 - d_1 z - \cfrac{ c_2 z}{
            1 - d_2 z - \cfrac{ c_3 z}{\ddots}
        }
    }
}.
\end{equation*}
\end{theo}

\begin{proof}
We can decompose a Schröder path as a a concatenation of smaller Schröder paths, by cutting every time the path reaches height $0$.
This shows $\sum D_{2n} z^n = (1-X)^{-1}$ where $X$ is the generating function of indecomposable paths.

If an indecomposable path is not a single step $(2,0)$ at height $0$, it begins with a step $(1,1)$, ends with a step $(1,-1)$,
and removing these two steps gives another Schröder path (of length $2$ less) with respect to a shifted origin. 
Since a step $(2,0)$ at height $0$ has weight $d_0$ and an ending step $(1,-1)$ has weight $c_1$, we have $X= d_0z + c_1zY$
where $Y$ is the weighted generating function of Schröder paths where the weights are appropriately shifted.

So we have
\begin{equation*}
\sum_{ n \geq 0} D_{2n} z^n =
\frac{1}{ 1 - d_0z -c_1 z Y}
\end{equation*}
and iterating the argument gives the continued fraction.
\end{proof}

\section{Specific notation}

Through the whole paper, we use two four-parameter families of integers defined by: 
\begin{align}
  \label{diffA} A^{n,k}_{i,j} &= \binom n{k+i}\binom n{k-j}-\binom n{k+i+1}\binom n{k-j-1}, \\[3mm]
  \label{diffB} B^{n,k}_{i,j} &= \binom{n+j-1}{k-1} \binom{n-i-1}{k-1} - \binom{n+j}{k-1} \binom{n-i-2}{k-1}.
\end{align}
By factorizing these expression, we get the alternative formulas:
\begin{align}
  \label{factA} A^{n,k}_{i,j} &= \frac{i+j+1}{n+1}  \binom{n+1}{k-j} \binom{n+1}{k+i+1}, \\[3mm]
  \label{factB} B^{n,k}_{i,j} &= \frac{i+j+1}{n+j}  \binom{n+j}{k-1} \binom{n-i-2}{k-2}.
\end{align}
The four parameters $n,k,i,j$ are non-negative integers such that $0\leq j \leq i \leq n-k$,
and we use the convention that $\binom{-2}{-2} = \binom{-1}{-1} = 1$, and $\binom{u}{v}=0$
if $v<0$ and $u>v$.

It is worth observing that we have:
\[
  A^{n,k}_{i,j} = A^{n,k-1}_{i+1,j-1}, \qquad B^{n,k}_{i,j} = B^{n+1,k}_{i+1,j-1},
\]
so that we could use three-parameter families of integers. But curiously, this observation does not seem
to be of any help in the present work. 

Also, we will use the {\it Narayana numbers} (see \href{https://oeis.org/A001263}{A001263} in \cite{sloane}), defined by:
\begin{align*}
   N_{n,k} = \frac{1}{n} \binom{n}{k-1} \binom{n}{k}
\end{align*}
if $1\leq k \leq n$. We slightly extend this definition by the convention $N_{0,0}=1$, and $N_{n,0}=0$ if $n\geq 1$.

\section{\texorpdfstring{$q$-Stirling numbers of the second kind}{q-Stirling numbers of the second kind}}

This section mostly presents results taken from \cite{rubeyjosuatverges}.

\begin{defi} 
 Let $\pi$ be a set partition of $\{1,\dots,n\}$. We denote $\cro(\pi)$ the number of {\it crossings} of $\pi$, 
 i.e.~pairs $((i,k),(j,l))$ with $i<j<k<l$ such that $i,k$ (respectively $j,l$) are two neighbour elements of a block of $\pi$.
\end{defi}

We represent a set partition by drawing arcs between each pair of neighbour elements in every block. Then $\cro(\pi)$
is seen to be the number of intersection points between the arcs. For example in the left part of Figure~\ref{crossings}, the set partition
$156|24|38|79{\rm A}$ has three crossings: $((2,4),(3,8))$, $((1,5),(3,8))$, and $((3,8),(7,9))$.       
       
\begin{defi}      
 Let $\delta_n$ denote the staircase Young diagram, with row lengths $n,n-1,\dots,1$ from bottom to top (using French notation).
 A {\it rook placement} is a (partial) filling of $\delta_n$ with some dots (also called rooks), such that there is at most 
 one dot in each row and in each column. Let $\inv(R)$ denote the number of {\it inversions} in a rook placement $R$, 
 i.e.~cells having a dot to its left in the same row and another below in the same column. 
\end{defi}

In a rook placement, we represent inversions by a cross. In the right part of Figure~\ref{crossings}, 
we have a rook placement with three inversions.

\begin{figure}[h!tp]
 \begin{tikzpicture}[scale=0.6]
   \tikzstyle{ver} = [circle, draw, fill, inner sep=0.5mm]
   \tikzstyle{edg} = [line width=0.6mm]
   \draw[color=white] (1,-4) rectangle (10,6);
   \node[ver] at (1,0) {};
   \node[ver] at (2,0) {};
   \node[ver] at (3,0) {};
   \node[ver] at (4,0) {};
   \node[ver] at (5,0) {};
   \node[ver] at (6,0) {};
   \node[ver] at (7,0) {};
   \node[ver] at (8,0) {};
   \node[ver] at (9,0) {};
   \node[ver] at (10,0) {};
   \node      at (1,-0.6) {1};
   \node      at (2,-0.6) {2};
   \node      at (3,-0.6) {3};
   \node      at (4,-0.6) {4};
   \node      at (5,-0.6) {5};
   \node      at (6,-0.6) {6};
   \node      at (7,-0.6) {7}; 
   \node      at (8,-0.6) {8}; 
   \node      at (9,-0.6) {9}; 
   \node      at (10,-0.6) {A}; 
   \draw[edg] (1,0) to[bend left=60] (5,0);
   \draw[edg] (5,0) to[bend left=60] (6,0);
   \draw[edg] (2,0) to[bend left=60] (4,0);
   \draw[edg] (3,0) to[bend left=60] (8,0);
   \draw[edg] (7,0) to[bend left=60] (9,0);
   \draw[edg] (9,0) to[bend left=60] (10,0);
 \end{tikzpicture}
 \hspace{2cm}
 \begin{tikzpicture}[scale=0.6]
  \tikzstyle{rec} = [line width=0.6mm]
  \tikzstyle{ver} = [circle, draw, fill, inner sep=0.5mm]
  \draw[rec]  (0,0) rectangle (9,1);
  \draw[rec] (0,0) rectangle (8,2);
  \draw[rec] (0,0) rectangle (7,3);
  \draw[rec] (0,0) rectangle (6,4);
  \draw[rec] (0,0) rectangle (5,5);
  \draw[rec] (0,0) rectangle (4,6);
  \draw[rec] (0,0) rectangle (3,7);
  \draw[rec] (0,0) rectangle (2,8);
  \draw[rec] (0,0) rectangle (1,9);
  \node      at (0.5,9.5) {1};
  \node      at (1.5,8.5) {2};
  \node      at (2.5,7.5) {3};
  \node      at (3.5,6.5) {4};
  \node      at (4.5,5.5) {5};
  \node      at (5.5,4.5) {6};
  \node      at (6.5,3.5) {7};
  \node      at (7.5,2.5) {8};
  \node      at (8.5,1.5) {9};
  \node      at (9.5,0.5) {A};
  \node[ver] at (0.5,5.5) {};
  \node[ver] at (1.5,6.5) {};
  \node[ver] at (2.5,2.5) {};
  \node[ver] at (4.5,4.5) {};
  \node[ver] at (6.5,1.5) {};
  \node[ver] at (8.5,0.5) {};
  \node      at (2.5,6.5) {$\times$};
  \node      at (2.5,5.5) {$\times$};
  \node      at (6.5,2.5) {$\times$};
 \end{tikzpicture}
 \caption{A set partition, and a rook placement. \label{crossings}}
\end{figure}
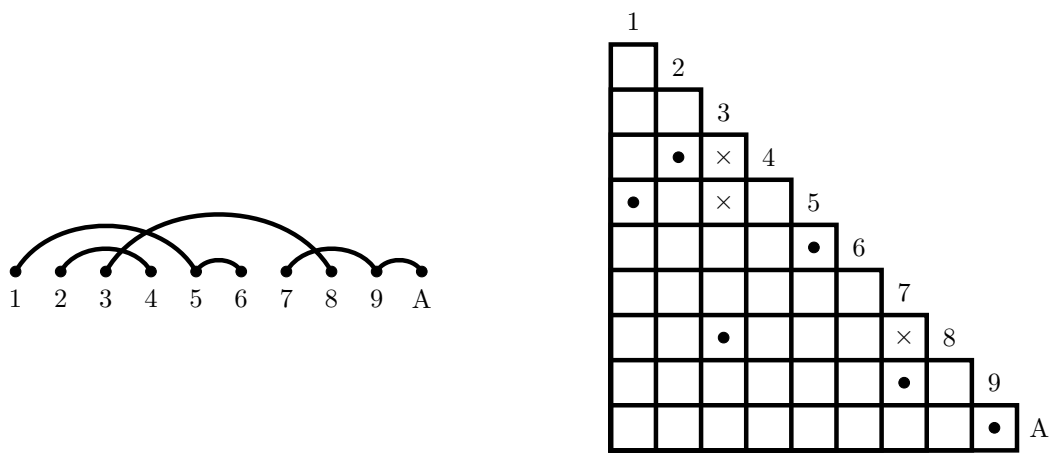

There is a simple bijection between set partitions of $\{1,\dots,n\}$ with $k$ blocks and rook placements
in $\delta_{n-1}$ with $n-k$ rooks. We first label the outer corners of $\delta_{n-1}$ with integers
from $1$ to $n$, from the top left corner to the bottom right one. Then, for each arc $(i,j)$ in the 
set partition, we put a cross at the intersection of the column having $i$ above it and the row having $j$ 
to its right. For example the set partition in the left part of Figure~\ref{crossings} corresponds to the rook placement 
to its right. It is easily seen that the crossings of a set partition are in bijection with the inversions of the 
corresponding rook placement.

It is well known that $S_2(n,k)$ is the number of set partitions of $\{1,\dots,n\}$ with $k$ blocks,
and the $q$-analog is given by:

\begin{defi} \label{defqst2}
If $0\leq k \leq n$, let $S_2[n,k] = \sum_{\pi} q^{\cro(\pi)}$ where the sum is over set partitions $\pi$
of $\{1,\dots,n\}$ with $k$ blocks. Alternatively, $S_2[n,k] = \sum_{R} q^{\inv(R)}$ where the sum
is over rook placements in $\delta_{n-1}$ with $n-k$ rooks.
\end{defi}

In \cite{rubeyjosuatverges}, we studied the notion of crossings in several combinatorial objects, 
from the point of view of orthogonal polynomials and their moments. 
In particular, by giving a bijection between set partitions and weighted Motzkin paths
and using Theorem~\ref{chemins}, we got:

\begin{theo}[cf.~\cite{rubeyjosuatverges}] \label{qfrac2}
\begin{equation*}
\sum_{0\leq k \leq n} S_2[n,k] x^k z^n =
\cfrac{1}{
    1 - xz - \cfrac{[1]_q x z^2}{
        1 - (x+[1]_q)z - \cfrac{[2]_qxz^2}{
            1 - (x+[2]_q)z - \cfrac{[3]_qxz^2}{\ddots}
        }
    }
}.
\end{equation*}
\end{theo}

And it is possible to extract a formula for each coefficient from the continued fraction:

\begin{theo}[cf.~\cite{rubeyjosuatverges}] \label{thformulaqst2}
\begin{equation} \label{formulaqst2}
  (1-q)^{n-k} S_2[n,k] =  \sum\limits_{j=0}^{k} \sum\limits_{i=j}^{n-k} (-1)^i 
  A^{n,k}_{i,j} q^{\tbinom{j+1}2} \qbin{i}{j}.
\end{equation}
\end{theo}

\begin{rema}
Although it is not needed for our purpose, let us mention that we found an alternative formula in the course of this work.
If $n>0$, we have:
\begin{equation} \label{formulaqst2bis}
  (1-q)^{n-k} S_2[n,k] =  \sum\limits_{j=0}^{k-1} \sum\limits_{i=j}^{n-k} (-1)^i 
  A^{n-1,k-1}_{i,j} q^{ i + \tbinom{j}2} \qbin{i}{j}.
\end{equation}
See Remark~\ref{altformulas} in the next section.
\end{rema}

\begin{rema}
Note that partitions with no crossing are the {\it noncrossing partitions} in the sense of Kreweras~\cite{kreweras}.
Since the number of noncrossing partitions of $\{1,\dots,n\}$ in $k$ blocks is $N_{n,k}$ (see Corollaire~4.1 in loc.~cit.), 
we have:
\begin{equation} \label{qst2q0}
   S_2[n,k] |_{q=0} = N_{n,k} .
\end{equation}
\end{rema}

\section{\texorpdfstring{$q$-Stirling numbers of the first kind}{q-Stirling numbers of the first kind}}

\begin{defi}
Let $\mathfrak{S}_n$ denote the group of permutations of $\{1,\dots,n\}$. 
A permutation is represented by its graph, a $n\times n$ grid where we put a dot at each cell with coordinates $(i,\sigma(i))$
for $1\leq i \leq n$.
Let $\sigma\in\mathfrak{S}_n$, then a {\it right-to-left maxima} of $\sigma$ is an integer $i\in\{1,\dots,n\}$ such that $\sigma(i)>\sigma(j)$ if $i<j\leq n$.
Let $\rlm(\sigma)$ denote the number of right-to-left maxima of $\sigma$.
The {\it bounding path} of $\sigma\in\mathfrak{S}_n$ is the path characterized by the following properties:
\begin{itemize}
 \item It starts at $(0,n)$ and ends at $(n,0)$. 
 \item It has $2n$ unit steps going East or South, in coordinates: $(1,0)$ or $(0,-1)$.
 \item All cells $(i,\sigma(i))$ are below it.
 \item It is as low as possible among paths satisfying the previous three properties. 
\end{itemize}
\end{defi}

In the left part of Figure~\ref{graphperm}, we have the graph of the permutation $869237514$ and its bounding path as a dashed line.
After the appropriate rotation so that the steps $(1,0)$ and $(0,-1)$ respectively become $(1,1)$ and $(1,-1)$, 
it is easily seen that the bounding path becomes a Dyck path of length $2n$, see the right of Figure~\ref{graphperm}.
Moreover, each right-to-left maxima in the permutation corresponds to a {\it peak} in its bounding path, 
i.e.~a step $(1,1)$ immediately followed by a step $(1,-1)$ to its right.

Let $\sigma\in\mathfrak{S}_n$. We recall that an {\it inversion} of $\sigma$ is a pair $(i,j)$ such that $1\leq i<j\leq n$ and 
$\sigma(i)>\sigma(j)$.

\begin{defi}
 A {\it special inversion} of a permutation $\sigma$ is an inversion $(i,j)$ such that there exists $k$
 satisfying $j<k\leq n$ and $\sigma(i)<\sigma(k)$. Let $\inv'(\sigma)$ denote the number of special inversions of $\sigma$.
\end{defi}

The inversion $(i,j)$ can be identified with the cell $(j,\sigma(i))$ (as in the previous section).
Then it is special if and only if this cell $(j,\sigma(i))$ is below the bounding path.
In this case, we represent this special inversion by putting a $\times$ in the cell $(j,\sigma(i))$. See the left part of 
Figure~\ref{graphperm} where there are 5 special inversions.

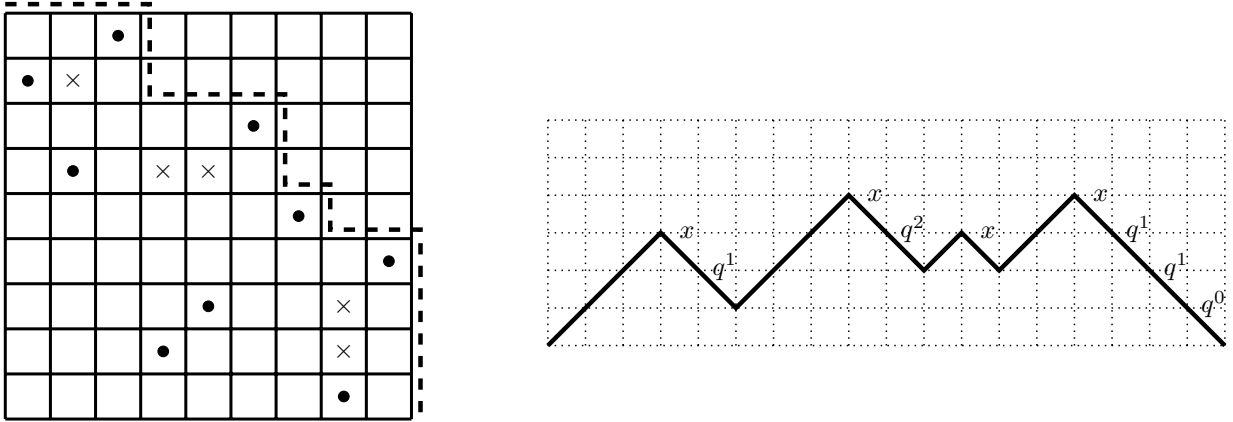
\begin{figure}
 \begin{tikzpicture}[scale=0.6]
   \tikzstyle{ver} = [circle, draw, fill, inner sep=0.5mm]
   \draw[line width=0.4mm] (0, 0) grid (9, 9);
   \node[ver] at (0.5,7.5) {};
   \node[ver] at (1.5,5.5) {};
   \node[ver] at (2.5,8.5) {};
   \node[ver] at (3.5,1.5) {};
   \node[ver] at (4.5,2.5) {};
   \node[ver] at (5.5,6.5) {};
   \node[ver] at (6.5,4.5) {};
   \node[ver] at (7.5,0.5) {};
   \node[ver] at (8.5,3.5) {};
   \node      at (1.5,7.5) {$\times$};
   \node      at (3.5,5.5) {$\times$};
   \node      at (4.5,5.5) {$\times$};
   \node      at (7.5,2.5) {$\times$};
   \node      at (7.5,1.5) {$\times$};
   \draw[line width = 0.6mm, dashed, dash pattern=on 1.5mm off 1.5mm] (0,9.2) -- (3.2,9.2) -- (3.2,7.2) -- (6.2,7.2) -- (6.2,5.2) -- (7.2,5.2) -- (7.2,4.2) 
         -- (9.2,4.2) -- (9.2,0);
 \end{tikzpicture}
 \hspace{1.4cm}
 \begin{tikzpicture}[scale=0.5]
   \draw[color=white,dashed] (0,-2) rectangle (18,8);
   \draw[line width=0.2mm,dotted] (0, 0) grid (18,6);
   \draw[line width = 0.6mm] (0,0) -- (3,3) -- (5,1) -- (8,4) -- (10,2) -- (11,3) -- (12,2) -- (14,4) -- (18,0);
   \node at (3.7,3) {$x$};
   \node at (4.7,2.1) {$q^1$};
   \node at (8.7,4) {$x$};
   \node at (9.7,3.1) {$q^2$};
   \node at (11.7,3) {$x$};
   \node at (14.7,4) {$x$};
   \node at (15.7,3.1) {$q^1$};
   \node at (16.7,2.1) {$q^1$};
   \node at (17.7,1.1) {$q^0$};
 \end{tikzpicture}
 \caption{Bounding path and special inversions of the permutation $869237514$ (left) 
          and a weighted Dyck path (right).\label{graphperm}}
\end{figure}

It is well known that $S_1(n,k)$ is the number of permutations $\sigma\in\mathfrak{S}_n$
having $k$ right-to-left maxima. The $q$-analog is given by:

\begin{defi} \label{defqst1}
Let $S_1[n,k] = \sum_{\sigma} q^{\inv'(\sigma)}$ where we sum over permutations $\sigma\in\mathfrak{S}_n$
having $k$ right-to-left maxima.
\end{defi}

There is also a continued fraction expansion for the generating function of $q$-Stirling numbers of the first kind,
but this time it is a T-fraction. The theorem below is a variant of the results in \cite[Section~3]{robletviennot}.

\begin{theo} \label{qfrac1}
\begin{equation} \label{eqqfrac1}
\sum_{0\leq k \leq n} S_1[n,k] x^k z^n =
\cfrac{1}{
    1 - (x-[1]_q)z - \cfrac{[1]_q z}{
        1 - (x-[2]_q)z - \cfrac{[2]_qz}{
            1 - (x-[3]_q)z - \cfrac{[3]_qz}{\ddots}
        }
    }
}.
\end{equation}
\end{theo}

\begin{proof}
Note that 
\[
  \sum_{k=0}^n S_1[n,k] x^k = \sum_{\sigma \in \mathfrak{S}_n} q^{\inv'(\sigma)} x^{\rlm(\sigma)}.
\]
To get the generating function of this quantity, the idea is to build a bijection $\phi$ between $\mathfrak{S}_n$ and the 
set $\mathcal{P}_n$ of weighted Dyck paths of length $2n$ such that:
\begin{itemize}
 \item a step $(1,-1)$ following a step $(1,1)$ has weight $x$,
 \item a step $(1,-1)$ following a step $(1,-1)$ and starting at height $i$ has a weight $q^j$ where $0\leq j \leq i-1$.
\end{itemize}
Moreover we require that the weight of $\phi(\sigma)$ is $ q^{\inv'(\sigma)} x^{\rlm(\sigma)} $.
The result then follows using Theorem~\ref{chemins2}.
This bijection is given in \cite[Section~3]{robletviennot}, but let us define it (for the convenience of the reader and
because we have different conventions).

Let $\sigma\in\mathfrak{S}_n$. First, the underlying Dyck path of $\phi(\sigma)$ is its bounding path. 
Then, the weight of the $i$th step $(1,-1)$ in the Dyck path is:
\begin{itemize}
 \item $x$ if this step follows a step $(1,1)$,
 \item $q^j$ where $j$ is the number of $\times$ in the $i$th row (from top to bottom) in the graph of $\sigma$, otherwise. 
\end{itemize}
See Figure~\ref{graphperm} for an example, where the weighted Dyck path on the right is the image of the permutation on the left. 
It is clear that the weight of $\phi(\sigma)$ is $ q^{\inv'(\sigma)} x^{\rlm(\sigma)} $.

To prove $\phi(\sigma)\in\mathcal{P}_n$, we need to check that the weights are valid. Consider the $i$th step $(1,-1)$ in the 
Dyck path. Let $h$ be its initial height, so that there are $ h+i-1$ steps $(1,1)$ to its left. Now if we consider the 
path in the graph of the permutation (as in the left part in Figure~\ref{graphperm}), we see that in the $i$th row, the number
of cells to the right of the bounding path is $n-(h+i-1)$. But all these cells are inversions (with the convention that 
inversions are identified with cells of the permutation graph as explained above) that are not special. 
The number of inversions in the $i$th row is at most $n-i$ so the number of special inversions in this row 
is smaller than $(n-i) - (n-h-i+1) = h-1$. It follows that $\phi(\sigma)\in\mathcal{P}_n$.

Now if we have a path in $\mathcal{P}_n$, its image under the inverse bijection is obtained by building the permutation graph
from top to bottom as follows. If the $i$th step $(1,-1)$ has weight $x$, the dot in the $i$th row is placed just to the left of the 
bounding path. If its weight is $q^j$, it is placed in the $(j+1)$th cell (from right to left) among cells to the left of the bounding path
that have no dot above in the same row. We omit details about why this is well defined and why this is the inverse bijection of $\phi$.
\end{proof}

We will use the combinatorics of paths to prove the formula below:

\begin{theo} \label{thformulaqst1}
\begin{equation} \label{formulaqst1}
  (1-q)^{n-k} S_1[n,k] = \sum\limits_{j=0}^{n-k} \sum\limits_{i=j}^{n-k} (-1)^j
  B^{n,k}_{i,j} q^{\tbinom{j+1}2} \qbin{i}{j}.
\end{equation}
\end{theo}

Alternatively:

\begin{theo} \label{thformulaqst1bis}
\begin{equation} \label{formulaqst1bis}
(1-q)^{n-k}S_1[n,k] =
\sum\limits_{j=0}^{n-k} \sum\limits_{i=j}^{n-k} (-1)^j
B^{n+1,k+1}_{i,j}
 q^{i + \tbinom{j}2} \qbin{i}{j}.
\end{equation}
\end{theo}

Our proof heavily relies on the combinatorics of Schröder and Motzkin paths.

%

\begin{lemm} \label{schropath}
The sum $ \sum_{k=0}^n (1-q)^{n-k}S_1[n,k] x^k$ is the generating function of weighted Schröder paths of length $2n$ such that:
\begin{itemize}
 \item the weight of a step $(1,-1)$ starting at height $i$ is $1-q^i$,
 \item the weight of a step $(2,0)$ at height $i$ is $x-1$ or $q^{i+1}$.
\end{itemize}
\end{lemm}

\begin{proof}
Making the substitution $x  \to x(1-q)^{-1} $ and $ z \to z(1-q)$ in \eqref{eqqfrac1} gives:
\begin{equation*}
\sum_{0\leq k \leq n} (1-q)^{n-k} S_1[n,k] x^k z^n =
\cfrac{1}{
    1 - (x-1+q)z - \cfrac{(1-q)z}{
        1 - (x-1+q^2)z - \cfrac{(1-q^2)z}{
            1 - (x-1+q^3)z - \cfrac{(1-q^3)z}{\ddots}
        }
    }
}.
\end{equation*}
So the result follows from Theorem~\ref{chemins3}.
\end{proof}

Note that in the previous lemma, instead of giving a weight $x-1+q^{h+1}$ to a step $(2,0)$ at height $h$,
we distinguish two possibilities. The reason will appear in the proof of the next lemma.

\begin{defi} \label{defmu}
For $k$, $n$ having the same parity and such that $0\leq k \leq n$,
let $\mu_{n,k}$ be the generating functions of weighted Motzkin paths of length $n$ such that:
\begin{itemize}
 \item there are $k$ horizontal steps,
 \item the weight of a step $(1,0)$ at height $h$  is $q^h$,
 \item the weight of a step $(1,-1)$ starting at height $h$ is $1-q^h$.
\end{itemize}
\end{defi}

\begin{lemm}
\begin{align} \label{qstsum}
  \sum_{k=0}^n (1-q)^{n-k} S_1[n,k] x^k = \sum_{i=0}^n \sum_{j=0}^{n-i} \binom{2n-i-j}{j} (x-1)^j q^i \mu_{2n-2j-i,i}.
\end{align}
\end{lemm}

\begin{proof}
We give a bijective proof.
Consider a weighted path $p$ as described in Lemma~\ref{schropath}.
Let $j$ be its number of steps $(2,0)$ with weight $x-1$, and $i$ its number of steps $(2,0)$ with weight $q^{h+1}$.
Then, define a weighted Motzkin path $\psi(p)$ by the following operations:
\begin{itemize}
 \item remove from $p$ each step $(2,0)$ with weight $x-1$,
 \item transform each step $(2,0)$ with weight $q^{h+1}$ into a step $(1,0)$ with weight $q^h$.
\end{itemize}
(Of course we need to connect all remaining steps to get a valid path.)
The length of $\psi(p)$ is $2n-2j-i$, it has $i$ steps $(1,0)$, and its weights are as in Definition~\ref{defmu}.
Moreover the weight of $p$ is $(x-1)^j q^i$ times that of $\psi(p)$.

To finish this bijective proof of \eqref{qstsum}, we need to check that $\binom{2n-i-j}{j}$ is the number 
of such paths $p$ as before that have the same image under $\psi$. Note that these paths $p$ have length $2n$ 
and $i+j$ steps $(2,0)$, so their total number of steps is $2n-i-j$. Hence, $\binom{2n-i-j}{j}$ is the 
number of possible locations of the steps $(2,0)$ with weight $x-1$, among all steps. 
But it is clear that we can recover $p$ from $\psi(p)$ and the location of these steps and we can conclude.
\end{proof}

\begin{lemm}
\[
  \mu_{n,k}  = \sum_{u=0}^{\frac{n-k}{2}} (-1)^u q^{\binom{u+1}{2}} \qbin{k+u}{u} \left( \binom{n}{ \frac{n-k}{2} - u } - \binom{n}{ \frac{n-k}{2} - u - 1} \right).
\]
\end{lemm}

\begin{proof}
From the definition in terms of weighted Motzkin paths, $\sum_{k=0}^n \mu_{n,k} a^k$ can be seen as the $n$th
moment of some sequence of orthogonal polynomials, called the continuous Big $q$-Hermite polynomials \cite{koekoekleskyswarttouw}.
They are obtained from Al-Salam-Chihara orthogonal polynomials \cite{koekoekleskyswarttouw} by setting some parameter to $0$.
Then the formula for $\mu_{n,k}$ can be obtained, for example, from \cite[Theorem~6.1.1]{josuat2}.
\end{proof}

\begin{proof}[Proof of Theorem~\ref{thformulaqst1bis}]
Taking the coefficient of $x^k$ in both sides of \eqref{qstsum}, we get:
\begin{align*}
  (1-q)^{n-k} S_1[n,k] &= \sum_{i=0}^{n-k} \sum_{j=k}^{n-i} \tbinom{2n-i-j}{j} (-1)^{j-k} \tbinom{j}{k} q^i \mu_{2n-2j-i,i}   \\
                      &= \sum_{i=0}^{n-k} \sum_{j=k}^{n-i} \tbinom{2n-i-j}{j} (-1)^{j-k} \tbinom{j}{k} q^i \sum_{u=0}^{n-j-i} (-1)^u q^{\binom{u+1}{2}}
                               \qbin{u+i}{u} \left(   \tbinom{2n-2j-i}{n-j-i-u} - \tbinom{2n-2j-i}{n-j-i-u-1}   \right) \\
                      &= \sum_{i=0}^{n-k} \sum_{u=0}^{n-i} q^i  (-1)^u q^{\binom{u+1}{2}} \qbin{u+i}{u} \sum_{j=k}^{n-i-u} \tbinom{2n-i-j}{j} (-1)^{j-k} \tbinom{j}{k} 
                                \left(   \tbinom{2n-2j-i}{n-j-i-u} - \tbinom{2n-2j-i}{n-j-i-u-1}   \right).
\end{align*}
Then the latter sum over $j$ is  $E_u-E_{u+1}$ where we defined:
\begin{align*}
  E_u  &= \sum_{j=k}^{n-i-u} \binom{2n-i-j}{j} (-1)^{j-k} \binom{j}{k} \binom{2n-2j-i}{n-j-i-u} \\
       &= \sum_{j=0}^{n-i-u-k} \binom{2n-i-k-j}{j+k} (-1)^j \binom{j+k}{k} \binom{2n-2j-2k-i}{n-j-k-i-u} \\
       &= \sum_{j=0}^{n-i-u-k} (-1)^j \frac{ (2n-i-k-j)! } { (2n-2k-2j-i)! j!k! } \binom{2n-2j-2k-i}{n-j-k-i-u}.
\end{align*}
And, using $(\alpha-j)! = (-1)^j \frac{\alpha!}{(-\alpha)_j}$, we have:
\begin{align*}
E_u &= \sum_{j=0}^{n-i-u-k} (-1)^j \frac{ (2n-i-k-j)! } { j!k! } \times\frac{1}{ (n-j-k-i-u)! (n-j-k+u )!} \\
    &= \frac{ (2n-i-k)! }{ (n-k-i-u)! (n-k+u )! k! } \sum_{j=0}^{n-i-u-k} \frac{ (-n+k+i+u)_j (-n+k-u)_j } {(-2n+i+k)_j j! } \\
    &= \frac{ (2n-i-k)! }{ (n-k-i-u)! (n-k+u )! k! }  \times \frac{ (-n+i+u)_{n-k-i-u}  }{ (-2n+i+k)_{n-k-i-u}  } \\
    &= \frac{ (2n-i-k)! }{ (n-k-i-u)! (n-k+u )! k! }  \times \frac{ (n-i-u)! (n+u)! }{ k! (2n-i-k)! } = \binom{n+u}{k} \binom{n-i-u}{k}.
\end{align*}
So, from \eqref{diffB} we get
\begin{equation*}
  E_u - E_{u+1} = B^{n+1,k+1}_{u+i,u}.
\end{equation*}
So, 
\begin{align*}
  (1-q)^{n-k} S_1[n,k] = \sum_{i=0}^{n-k} \sum_{u=0}^{n-i} B^{n+1,k+1}_{u+i,u} (-1)^u q^{ i + \binom{u+1}{2}} \qbin{u+i}{u}.
\end{align*}
We can restrict the summation on indices such that $u+i\leq n-k$ since otherwise $B^{n+1,k+1}_{u+i,u}=0$.
This gives \eqref{formulaqst1bis}.
\end{proof}

\begin{proof}[Proof of Theorem~\ref{thformulaqst1}]
We show that \eqref{formulaqst1bis} implies \eqref{formulaqst1}.
Using the recursion for $q$-binomial coefficient in \eqref{recqbin} a couple of times, we have:
\begin{align*}
   q^{i + \binom{j}{2} } \qbin{i}{j} &= q^{\binom{j+1}{2}} \qbin{i+1}{j+1} - q^{\binom{j+1}{2}} \qbin{i}{j+1}  \\
       &= q^{\binom{j+2}{2}} \qbin{i}{j+1} + q^{\binom{j+1}{2}} \qbin{i}{j} - q^{\binom{j+2}{2}} \qbin{i-1}{j+1} - q^{\binom{j+1}{2}} \qbin{i-1}{j}.
\end{align*}
So:
\begin{align*}
   (1-q)^{n-k} S_1[n,k] &= \sum_{ 0  \leq j \leq i \leq n-k } B^{n+1,k+1}_{i,j} (-1)^j q^{i + \binom{j}{2} } \qbin{i}{j}  \\
      &=   \sum_{ 0  \leq j \leq i \leq n-k } B^{n+1,k+1}_{i,j} (-1)^j q^{\binom{j+2}{2}} \qbin{i}{j+1}
         + \sum_{ 0  \leq j \leq i \leq n-k } B^{n+1,k+1}_{i,j} (-1)^j q^{\binom{j+1}{2}} \qbin{i}{j}  - \\
       &   \sum_{ 0  \leq j \leq i \leq n-k } B^{n+1,k+1}_{i,j} (-1)^j q^{\binom{j+2}{2}} \qbin{i-1}{j+1}
         - \sum_{ 0  \leq j \leq i \leq n-k } B^{n+1,k+1}_{i,j} (-1)^j q^{\binom{j+1}{2}} \qbin{i-1}{j}  \\
      &=   \sum_{ 1  \leq j \leq i \leq n-k } B^{n+1,k+1}_{i,j-1} (-1)^{j+1} q^{\binom{j+1}{2}} \qbin{i}{j}
         + \sum_{ 0  \leq j \leq i \leq n-k } B^{n+1,k+1}_{i,j} (-1)^j q^{\binom{j+1}{2}} \qbin{i}{j}  - \\
       &   \sum_{ 1  \leq j \leq i < n-k } B^{n+1,k+1}_{i+1,j-1} (-1)^{j+1} q^{\binom{j+1}{2}} \qbin{i}{j}
         - \sum_{ 0  \leq j \leq i < n-k } B^{n+1,k+1}_{i+1,j} (-1)^j q^{\binom{j+1}{2}} \qbin{i}{j} .
\end{align*}
Note that in the last two sums, we can include the terms where $i=n-k$, since $B^{n+1,k+1}_{n-k+1,j}= 0$ (the second binomial vanishes in \eqref{factB}).
We can also include the terms where $j=0$ in the first and third sums, since they cancel each other out (to do that we also need to check from 
the formula that $B^{n+1,k+1}_{0,-1}=0$ and $B^{n+1,k+1}_{n-k+1,-1}= 0$).
Eventually we get:
\begin{align*}
    (1-q)^{n-k} S_1[n,k] = \sum_{ 0  \leq j \leq i \leq n-k } (-B^{n+1,k+1}_{i,j-1}+B^{n+1,k+1}_{i,j}+B^{n+1,k+1}_{i+1,j-1}-B^{n+1,k+1}_{i+1,j}) (-1)^j q^{\binom{j+1}{2}} \qbin{i}{j}. 
\end{align*}
To obtain \eqref{formulaqst1}, it remains to prove:
\begin{align} \label{rel4B}
   -B^{n+1,k+1}_{i,j-1}+B^{n+1,k+1}_{i,j}+B^{n+1,k+1}_{i+1,j-1}-B^{n+1,k+1}_{i+1,j} = B^{n,k}_{i,j}.
\end{align}
But this is easy to do by elementary manipulation on binomial coefficients. We omit details.
\end{proof}

%

\begin{rema} \label{altformulas}
It is possible to link \eqref{formulaqst2} and \eqref{formulaqst2bis}, in the same way that we linked
\eqref{formulaqst1} and \eqref{formulaqst1bis}. The related identity on the coefficients $A^{n,k}_{i,j}$ is 
\begin{align} \label{rel4A}
   A^{n-1,k-1}_{i,j-1}+A^{n-1,k-1}_{i,j}+A^{n-1,k-1}_{i+1,j-1}+A^{n-1,k-1}_{i+1,j} = A^{n,k}_{i,j}.
\end{align}
\end{rema}
Despite the similarity, we have got no clear link between \eqref{rel4A} and \eqref{rel4B}.

\begin{rema}
We have seen that the bijection $\phi$ from the proof of Theorem~\ref{qfrac1} gives an interpretation 
of $S_1[n,k]$ in terms of weighted Dyck paths of length $2n$ with $k$ peaks. It follows that $S_1[n,k]|_{q=0}$
is the number of such weighted Dyck paths  where all weights are $1$.
It is well known that the number of Dyck paths of length $2n$ with $k$ peaks is $N_{n,k}$, so that:
\begin{equation} \label{qst1q0}
  S_1[n,k]|_{q=0} = N_{n,k}.
\end{equation}
\end{rema}

\section{The first identity}

\begin{theo} \label{thqid1}
If $0\leq k \leq n$, we have:
\begin{equation} \label{qid1}
   S_1[n,k] = (-1)^{n-k} \sum_{j=0}^{n-k} (-1)^j
   \binom{n-1+j}{n-k+j} \binom{2n-k}{n-k-j} S_2[n-k+j,j].
\end{equation} 
\end{theo}

\begin{proof}
We compute the right-hand side of \eqref{qid1} using \eqref{formulaqst2}, and simplify the result to get
the right-hand side of \eqref{formulaqst1}. Up to the simple factor $(q-1)^{n-k}$, this gives: 
\begin{align*}
    (1-q)^{n-k} \sum_{j=0}^{n-k} (-1)^j 
   \binom{n-1+j}{n-k+j} & \binom{2n-k}{n-k-j} S_2[n-k+j,j] \\
    &= \sum_{j=0}^{n-k} (-1)^j
   \binom{n-1+j}{n-k+j} \binom{2n-k}{n-k-j} 
    \sum\limits_{i=0}^{j} \sum\limits_{h=i}^{n-k} (-1)^h
 A^{n-k+j,j}_{h,i}
 q^{\tbinom{i+1}2} \qbin h i  \\
 &= \sum_{h=0}^{n-k} \sum_{i=0}^h (-1)^h D^{n,k}_{h,i} q^{\tbinom{i+1}2} \qbin h i ,
\end{align*}
where $D^{n,k}_{h,i}$, the sum to be simplified, is defined by 
\begin{align*}
 D^{n,k}_{h,i} =
 \sum_{j=i}^{n-k} (-1)^j
 \binom{n-1+j}{n-k+j} \binom{2n-k}{n-k-j}
 A^{n-k+j,j}_{h,i}.
\end{align*}
To check the exchange of summation, note that the range of indices $h,i,j$ is defined by $0\leq h,i,j \leq n-k$ and $i\leq h,j$.
To finish the proof, it remains only to show:
\begin{align} \label{dequalb}
 D^{n,k}_{h,i} = (-1)^{n+k+h+i} B^{n,k}_{h,i}.
\end{align}
We rewrite the sum as an hypergeometric series, using the factorized formula for $A^{n-k+j,j}_{h,i}$ in Equation~\eqref{factA}:
\begin{align*}
   D^{n,k}_{h,i} &= \sum_{j=i}^{n-k} (-1)^j 
                           \frac{(n-1+j)!}{(n-k+j)!(k-1)!} \times \frac{(2n-k)! }{ (n-k-j)!(n+j)! }  \times \frac{h+i+1}{n-k+j+1} \\
                           & \hspace{5.5cm} \times \frac{(n-k+j+1)!}{(j-i)!(n-k+i-1)!} \times \frac{(n-k+j+1)!}{(j+h+1)!(n-k-h)! }  \\
   &= \frac{ (h+i+1) (2n-k)!}{(k-1)!(n-k-h)!(n-k+i+1)!}\sum_{j=i}^{n-k} (-1)^j \frac{ (n-1+j)!  (n-k+j+1)!  }{   (n-k-j)! (n+j)! (j+h+1)!  (j-i)!  } \\
   &= \frac{ (h+i+1) (2n-k)!}{(k-1)!(n-k-h)!(n-k+i+1)!}\sum_{j\geq 0} (-1)^{j+i} \frac{ (n-1+j+i)!  (n-k+j+i+1)!  }{   (n-k-j-i)! (n+j+i)! (j+i+h+1)!  j!  }.
\end{align*}
Then, we use $(m)_j = \frac{(m+j-1)!}{(m-1)!}$ if $m\geq 0$ to find:
\begin{align*}
   D^{n,k}_{h,i} &= \frac{ (h+i+1) (2n-k)! (n-1+i)!}{(k-1)!(n-k-h)!}\sum_{j\geq 0} (-1)^{j+i} \frac{ (n+i)_j  (n-k+i+2)_j  }{   (n-k-j-i)! (n+j+i)! (j+i+h+1)!  j!  } \\
   &= \frac{ (h+i+1) (2n-k)! }{ (n+i) (k-1)!(n-k-h)!}\sum_{j\geq 0} (-1)^{j+i} \frac{ (n+i)_j  (n-k+i+2)_j  }{   (n-k-j-i)! (n+i+1)_j (j+i+h+1)!  j!  } \\
   &= \frac{ (2n-k)! }{ (n+i) (k-1)!(n-k-h)! (h+i)! }\sum_{j\geq 0} (-1)^{j+i} \frac{ (n+i)_j  (n-k+i+2)_j  }{   (n-k-j-i)! (n+i+1)_j (i+h+2)_j  j!  }.
\end{align*}
For the last step, we use $(-m)_j = (-1)^j \frac{m!}{(m-j)!} $ to find:
\begin{align} \label{Dhyperg}
   D^{n,k}_{h,i} = \frac{ (-1)^{i} (2n-k)! }{ (n+i) (k-1)!(n-k-h)! (h+i)! (n-k-i)! }\sum_{j\geq 0}  \frac{ (k+i-n)_j (n+i)_j  (n-k+i+2)_j  }{    (n+i+1)_j (i+h+2)_j  j!  }.
\end{align}


The Pfaff-Saalschütz summation \eqref{pfaffsaalschutz} can be applied if $h=i$, since then 
the sum over $j$ becomes:
\begin{align*}
   {}_3 F_2 \left( \begin{array}{c} k-n+i \, ; \, n-k+2+i\, ; \, n+i \\ n+i+1 \, ; \, 2i+2 \end{array}  \Big| 1 \right).
\end{align*}
It follows:
\begin{align*}
   D^{n,k}_{i,i} &= \frac{ (-1)^{i} (2n-k)! }{ (n+i) (k-1)!(n-k-i)! (2i)! (n-k-i)! } \times 
   \frac{  (k-1)_{n-k-i}  (1)_{n-k-i} }{ (n+i+1)_{n-k-i} (k-i-n-1)_{n-k-i} }  \\
   &=  
   \frac{ (-1)^{i} (n+i)! }{ (n+i) (k-1)!(n-k-i)! (2i)! } \times 
   \frac{  (k-1)_{n-k-i}   }{  (k-i-n-1)_{n-k-i} }  \\   
   &=  
   \frac{ (-1)^{i} (n+i)! }{ (n+i) (k-1)!(n-k-i)! (2i)! } \times 
    \frac{ (n-i-2)! }{ (k-2)!} \times 
   \frac{  (-1)^{n-k-i}  (2i+1)!  }{ (n-k+i+1)! } \\   
   &=
    (-1)^{n-k}\frac{2i+1}{n+i} \times \frac{ (n+i)! }{(k-1)! (n+i-k+1)! } \times \frac{ (n-i-2)!}{ (n-k-i)! (k-2)! } = (-1)^{n-k} B^{n,k}_{i,i},
\end{align*}
using the factorized formula in \eqref{factB}.
%
%
So \eqref{dequalb} holds when $h=i$. 

Then, we give a simple relation between $D^{n,k}_{h,i}$ and $D^{n,k}_{h,i+1}$,
and check that the same holds between $B^{n,k}_{h,i}$ and $B^{n,k}_{h,i+1}$ (up to a minus sign).
In Equation~\eqref{Dhyperg}, we can use $\frac{(n+i)_j}{(n+i+1)_j} = \frac{n+i}{n+i+j} = 1 - \frac{j}{n+i+j}$, to get: 
\begin{align*}
   D^{n,k}_{h,i} &=  \frac{ (-1)^{i} (2n-k)! }{ (n+i) (k-1)!(n-k-h)! (h+i)! (n-k-i)! }\sum_{j\geq 0}  \frac{ (k+i-n)_j (n+i) (n-k+i+2)_j  }{  (n+i+j) (i+h+2)_j  j!  } \\
                 &=  \frac{ (-1)^{i} (2n-k)! }{ (n+i) (k-1)!(n-k-h)! (h+i)! (n-k-i)! } \times \\ 
                 & \hspace{3cm}  \left( \sum_{j\geq 0}  \frac{ (k+i-n)_j (n-k+i+2)_j  }{  (i+h+2)_j  j!  } 
                            - \sum_{j\geq 1}  \frac{ (k+i-n)_j (n-k+i+2)_j  }{  (n+i+j) (i+h+2)_j  (j-1)!  }  \right).
\end{align*}
The first sum over $j$ can be evaluated using \eqref{gauss}. It is equal to:
\begin{align*}
    \frac{  (h+k-n)_{n-k-i}  }{  (i+h+2)_{n-k-i}  }.
\end{align*}
The numerator is the product $(h+k-n)(h+k-n+1)\dots (h-i-1)$. Since $h+k-n\leq 0$, this is $0$ if $h-i-1\geq 0$ i.e.~$i+1\leq h$
(which we can assume since we want a relation between $D^{n,k}_{h,i}$ and $D^{n,k}_{h,i+1}$).
As for the second sum, using $(\alpha)_{j+1} = \alpha (\alpha+1)_j$, we have:
\begin{align*}
    \sum_{j\geq 1}  \frac{ (k+i-n)_j (n-k+i+2)_j  }{  (n+i+j) (i+h+2)_j  (j-1)!  }
     & = \tfrac{(k+i-n)(n-k+i+2)}{(i+h+2)} \times  \sum_{j\geq 0}  \frac{ (k+i-n+1)_j (n-k+i+3)_j  }{  (n+i+j+1) (i+h+3)_j  j!  } \\
     & = \tfrac{(k+i-n)(n-k+i+2)}{(i+h+2)(n+i+1)} \times  \sum_{j\geq 0}  \frac{  (k+i-n+1)_j (n+i+1)_j (n-k+i+3)_j  }{  (n+i+2)_j (i+h+3)_j  j!  }.  
\end{align*}
Denote $S$ the latter sum, then we have:
\begin{align*}
  D^{n,k}_{h,i} = \frac{ (-1)^{i} (2n-k)! }{ (n+i) (k-1)!(n-k-h)! (h+i)! (n-k-i)! } \times \frac{(k+i-n)(n-k+i+2)}{(i+h+2)(n+i+1)} \times  S,
\end{align*}
and from \eqref{Dhyperg}:
\begin{align*}
  D^{n,k}_{h,i+1} = \frac{ (-1)^{i+1} (2n-k)! }{ (n+i+1) (k-1)!(n-k-h)! (h+i+1)! (n-k-i-1)! }  \times  S.
\end{align*}
It follows:
\begin{equation} \label{relDi}
 (n+i)(i+h+2) D^{n,k}_{h,i}  =  - (h+i+1) (n-k+i+2) D^{n,k}_{h,i+1}.
\end{equation}
From \eqref{factB} we also have:
\begin{align} \label{relBi}
  \frac{B^{n,k}_{h,i}}{B^{n,k}_{h,i+1}} = \frac{(h+i+1)(n+i+1)\binom{n+i}{k-1}}{ (h+i+2)(n+i)\binom{n+i+1}{k-1} }
        = \frac{ (h+i+1)(n+i-k+2) }{ (h+i+2)(n+i) }.
\end{align}
Equations~\eqref{relDi} and \eqref{relBi} implies that \eqref{dequalb} also holds in the range $0 \leq i \leq h$, since we have
already seen that it holds when $h=i$. This completes the proof.
\end{proof}

Note that the particular case $q=0$ in the previous theorem, together with \eqref{qst2q0} and \eqref{qst1q0}, 
gives the following identity (that will be used in the next section).

\begin{coro} If $0\leq k \leq n$, we have:
\begin{equation} \label{idnar}
  N_{n,k} = (-1)^{n-k} \sum_{j=0}^{n-k} (-1)^j \binom{n-1+j}{n-k+j} \binom{2n-k}{n-k-j} N_{n-k+j,j}. \\
\end{equation}
\end{coro}

\section{The second identity}

\begin{theo} \label{thqid2}
If $0\leq k \leq n$, we have:
\begin{equation} \label{qid2}
   S_2[n,k] = (-1)^{n-k} \sum_{j=0}^{n-k} (-1)^j
   \binom{n-1+j}{n-k+j} \binom{2n-k}{n-k-j} S_1[n-k+j,j].
\end{equation} 
\end{theo}

\begin{proof}
The scheme of proof is the same as in the previous section.
Using \eqref{formulaqst1}, the right hand side of \eqref{qid2} is
(up to the simple factor $(q-1)^{n-k}$ ):
\begin{align*}
    (1-q)^{n-k} \sum_{j=0}^{n-k} (-1)^j
   \binom{n-1+j}{n-k+j} & \binom{2n-k}{n-k-j} S_1[n-k+j,j] \\
   &  = \sum_{j=0}^{n-k} (-1)^j
   \binom{n-1+j}{n-k+j} \binom{2n-k}{n-k-j} 
    \sum\limits_{i=0}^{n-k} \sum\limits_{h=i}^{n-k} (-1)^i
  B^{n-k+j,j}_{h,i}
 q^{\tbinom{i+1}2} \qbin h i  \\
 &= \sum_{h=0}^{n-k} \sum_{i=0}^h (-1)^i  C^{n,k}_{h,i} q^{\tbinom{i+1}2} \qbin h i ,
\end{align*}
where $C^{n,k}_{h,i}$ is defined by:
\begin{align*}
 C^{n,k}_{h,i} &= \sum_{j=0}^{n-k} (-1)^j  \binom{n-1+j}{n-k+j} \binom{2n-k}{n-k-j} B^{n-k+j,j}_{h,i}.
\end{align*}
To obtain the right hand side of \eqref{formulaqst2} and finish the proof, we need to prove:
\begin{align} \label{sumcsuma}
        \sum_{h=0}^{n-k} \sum_{i=0}^h (-1)^i  C^{n,k}_{h,i} q^{\tbinom{i+1}2} \qbin h i 
 =
  (-1)^{n-k} \sum_{h=0}^{n-k} \sum_{i=0}^h (-1)^h  A^{n,k}_{h,i} q^{\tbinom{i+1}2} \qbin h i .
\end{align}
We will show that in the case $i>0$, we have:
\begin{align} \label{cequala}
 C^{n,k}_{h,i} = (-1)^{n+k+h+i} A^{n,k}_{h,i}.
\end{align}
Even if this does not hold for $i=0$, this is overcome as follows. We note that in both sides of \eqref{sumcsuma}, the terms where
$i=0$ are constant terms. So, after proving \eqref{cequala} for $i>0$, it remains only to check that both constant terms of
\eqref{sumcsuma} agree, or equivalently, that both constant terms of \eqref{qid2} agree. But this follows from \eqref{idnar}
since our $q$-Stirling numbers at $q=0$ are the Narayana numbers, see \eqref{qst2q0} and \eqref{qst1q0}.

So, let us prove \eqref{cequala} for $i>0$. To do this, we also define
\begin{align*}
 \bar C ^{n,k}_{h,i} &= \sum_{j=0}^{n-k} (-1)^j  \binom{n-1+j}{n-k+j} \binom{2n-k}{n-k-j} 
       \binom{n-k+j+i-1}{j-1} \binom{n-k+j-h-1}{j-1}.
\end{align*}
so that, from \eqref{diffA}, we have:
\begin{align*}
  C ^{n,k}_{h,i} = \bar C ^{n,k}_{h,i} - \bar C ^{n,k}_{h+1,i+1}.
\end{align*}
Then, it suffices to prove
\begin{align} \label{barCeqbins}
 \bar C ^{n,k}_{h,i} =  (-1)^{n+k+h+i} \binom{n}{k+h} \binom{n}{k-i},
\end{align}
since \eqref{cequala} follows, using \eqref{diffB}. Note that we need to define $\bar C ^{n,k}_{h,i}$
in the case $h=n-k+1$ and check that this is $0$, and this is easily done because $\binom{n-k+j-h-1}{j-1} = 0$ then.
So we keep assuming $h\leq n-k$.

We rewrite the sum for $\bar C ^{n,k}_{h,i}$ in hypergeometric form:
\begin{align*}
 \bar C ^{n,k}_{h,i} &= \sum_{j=0}^{n-k} (-1)^j  \frac{ (n-1+j)! (2n-k)! (n-k+j+i-1)! (n-k+j-h-1)! }
         { (n-k+j)!(k-1)! (n-k-j)! (n+j)! (j-1)!^2 (n-k+i)! (n-k-h)! } \\
     &= - \sum_{j\geq 0} (-1)^j  \frac{ (n+j)! (2n-k)! (n-k+j+i)! (n-k+j-h)! }
         { (n-k+j+1)!(k-1)! (n-k-j-1)! (n+j+1)! j!^2 (n-k+i)! (n-k-h)! } \\
     &= - \frac{(2n-k)!}{(k-1)!} \sum_{j\geq 0} (-1)^j  \frac{ (n+j)! (n-k+i+1)_j (n-k+j-h)! }
         { (n-k+j+1)! (n-k-j-1)! (n+j+1)!  (n-k-h)! j!^2 } \\
     &= - \frac{(2n-k)!}{(k-1)!} \sum_{j\geq 0} (-1)^j  \frac{ (n+j)! (n-k+i+1)_j (n-k-h+1)_j }
         { (n-k+j+1)! (n-k-j-1)! (n+j+1)! (1)_j j! } \\
     &= - \frac{(2n-k)!}{(k-1)!(n-k+1)!} \sum_{j\geq 0} (-1)^j  \frac{ (n+j)! (n-k+i+1)_j (n-k-h+1)_j }
         { (n-k+2)_j (n-k-j-1)! (n+j+1)! (1)_j j! } \\
     &= - \frac{(2n-k)!}{(n+1)(k-1)!(n-k+1)!} \sum_{j\geq 0} (-1)^j  \frac{ (n+1)_j (n-k+i+1)_j (n-k-h+1)_j }
         { (n-k+2)_j (n-k-j-1)! (n+2)_j (1)_j j! }.
\end{align*}
For the last step, use $(-m)_j = (-1)^j \frac{m!}{(m-j)!} $ to find:
\begin{align} \label{barChypergeo}
  \bar C ^{n,k}_{h,i} = - \tfrac{ (2n-k)! }{ (n+1) (n-k-1)!  (n-k+1)! (k-1)! }
  {}_4 F_3 \left( \begin{array}{c} k-n+1 \, ; \,  n+1 \, ; \, n-k-h+1 \, ; \, n-k+i+1 \\ 1 \, ; \, n+2 \, ; \, n-k+2 \end{array}  \Big| 1 \right).
\end{align}
Note that $h$ and $i$ only appear in the argument of the hypergeometric sum. So we can use Equation~\eqref{eqcontig}
to find a relation between $\bar C^{n,k}_{h,i}$ and the same with shifted indices, as follows:
\begin{align} \label{relcontig}
  -(h+i) \bar C^{n,k}_{h,i} = (n-k-h+1) \bar C^{n,k}_{h-1,i} - (n-k+i+1) \bar C^{n,k}_{h,i+1}.
\end{align}
On the other side, we have:
\begin{align*}
    (n-k-h+1) \binom{n}{k+h-1} \binom{n}{k-i} - & (n-k+i+1) \binom{n}{k+h}\binom{n}{k-i-1} \\
   &= (k+h) \binom{n}{k+h}\binom{n}{k-i} - (k-i) \binom{n}{k+h}\binom{n}{k-i} \\
   &=(h+i) \binom{n}{k+h}\binom{n}{k-i}.
\end{align*}
This means that the right hand side of \eqref{barCeqbins} satisfy the same relation as $\bar C^{n,k}_{h,i}$ in \eqref{relcontig}.
So, proceeding by induction on $h$, we are left to prove \eqref{barCeqbins} in the initial case $h=n-k$, since
the case $0\leq h \leq n-k$ follows. Note that in \eqref{relcontig}, the factor in front 
of $\bar C^{n,k}_{h-1,i}$ does not vanish in this range.

From \eqref{barChypergeo}, we have:
\begin{align} \label{barChypergeo2}
  \bar C^{n,k}_{n-k,i} 
  =
  -\frac{ (2n-k)! }{ (n+1) (n-k-1)! (n-k+1)! (k-1)! }
  {}_3 F_2 \left( \begin{array}{c} k-n+1 \, ; \, n+1 \, ; \, i-k+n+1 \\ n-k+2 \, ; \, n+2 \end{array}  \Big| 1 \right),
\end{align}
and our goal is to prove 
\begin{align} \label{Cbarbin}
  \bar C^{n,k}_{n-k,i} 
  =
  (-1)^i \binom{n}{k-i}. 
\end{align}
We use \eqref{eqcontig} once again, to find:
\begin{align*}
 (k-i) \bar C^{n,k}_{n-k,i}  =  
    \tfrac{ (2n-k)! }{ (n-k-1)! (n-k+1)! (k-1)! }
    {}_2 F_1 \left( \begin{array}{c} k-n+1 \, ; \, i-k+n+1 \\ n-k+2  \end{array}  \Big| 1 \right)
   -  (n-k+i+1) \bar C^{n,k}_{n-k,i+1}.
\end{align*}
The series can be evaluated using~\eqref{gauss}:
\begin{align*}
  {}_2 F_1 \left( \begin{array}{c} k-n+1 \, ; \, i-k+n+1 \\ n-k+2  \end{array}  \Big| 1 \right)
 =
  \frac{ (-i+1)_{n-k-1} }{ (n-k+2)_{n-k-1}  }.
\end{align*}
The numerator is the product $(-i+1)\dots (n-k-i-1)$, so it is $0$ if $i\geq 1$ and $i\leq n-k-1$.
So, in the range $1\leq i \leq n-k-1 $, we have:
\begin{align*}
  (k-i) \bar C^{n,k}_{n-k,i} = - (n-k+i+1) \bar C^{n,k}_{n-k,i+1}.
\end{align*}
Since the right hand side of \eqref{Cbarbin} satisfies the same relation, it remains only to prove \eqref{Cbarbin}
for the special value $i=1$. From \eqref{barChypergeo2}, we have:
\begin{align*}
  \bar C^{n,k}_{n-k,1} 
  & =
    -\frac{ (2n-k)! }{ (n+1) (n-k-1)! (n-k+1)! (k-1)! }
    {}_2 F_1 \left( \begin{array}{c} k-n+1 \, ; \, n+1 \\  n+2 \end{array}  \Big| 1 \right) \\
  & =
    -\frac{ (2n-k)! }{ (n+1) (n-k-1)! (n-k+1)! (k-1)! }
    \times \frac{ (1)_{n-k-1} }{ (n+2)_{n-k-1} } \\
  & =
    -\frac{ (2n-k)! }{ (n+1)  (n-k+1)! (k-1)! (n+2)_{n-k-1} } 
    = 
    -\frac{n!}{(n-k+1)! (k-1)!} = - \binom{n}{k-1}.
\end{align*}
This completes the proof of \eqref{Cbarbin}, hence of \eqref{barCeqbins}, hence of \eqref{cequala}, hence of \eqref{sumcsuma},
and the result follows.
\end{proof}

%
%
%
%
%

\renewcommand{\arraystretch}{2}

\begin{table}[h!tp] \centering  \tiny

\begin{sideways}
 $S_1[n,k]$ 
\end{sideways}
\hspace{6mm}
\begin{sideways}
\begin{tabular}{c|c|c|c|c|c|c|c|c|}
 $k \backslash n$  & 0 & 1 & 2 & 3 & 4 & 5 & 6 & 7 \\[1mm] \hline
                0  & 1 & . & . & . & . & . & . & . \\[1mm] \hline
                1  & 0 & 1 & . & . & . & . & . & . \\[1mm] \hline
                2  & 0 & 1 & 1 & . & . & . & . & . \\[1mm] \hline
                3  & 0 & $q+1$ & 3 & 1 & . & . & . & . \\[1mm] \hline
                4  & 0 & $q^3 + 2q^2 + 2q + 1$ & $5q+6$ & 6 & 1 & . & . & . \\[1mm] \hline
                5  & 0 & $q^6 + 3q^5 + 5q^4 + 6q^3 + 5q^2 + 3q + 1$ & $7q^3 + 15q^2 + 18q + 10$ & $15q+20$ & 10 & 1 & . & . \\[1mm] \hline
                6  & 0 & $q^{10} + 4q^9 + 9q^8 + 15q^7 + 20q^6  + $ & $9q^6 + 28q^5 + 50q^4 + $ 
                    & $28q^3 + 63q^2 + 84q + 50$ & $35q + 50$ & 15 & 1 & . \\
                   &   & $22q^5 + 20q^4 + 15q^3 + 9q^2 + 4q + 1$ & $67q^3 + 63q^2 + 42q + 15 $ & & & & & \\[1mm] \hline
                7  & 0 & $q^{15} + 5q^{14} + 14q^{13} + 29q^{12} + 49q^{11} + 71q^{10} + 90q^9 +   $ & $11q^{10} + 45q^9 + 105q^8 + 184q^7 + 264q^6 +  $ 
                    & $45q^6 + 144q^5 + 270q^4 +  $ & $84q^3 + 196q^2 +  $ & $70q$ & 21 & 1 \\ 
                   &   & $101q^8 + 101q^7 + 90q^6 + 71q^5 + 49q^4 + 29q^3 + 14q^2 + 5q + 1$ & $315q^5 + 313q^4 + 258q^3 + 168q^2 + 80q + 21$ & $388q^3 + 392q^2 + 280q + 105 $ & $280q + 175$ & $+105$ & & \\
\end{tabular} 
\end{sideways}
\hspace{2cm}
\begin{sideways}
$S_2[n,k]$ 
\end{sideways}
\hspace{6mm}
\begin{sideways}
\begin{tabular}{c|c|c|c|c|c|c|c|c|}
 $k \backslash n$  & 0 & 1 & 2 & 3 & 4 & 5 & 6 \\[1mm] \hline
                0  & 1 & . & . & . & . & . & . \\[1mm] \hline
                1  & 0 & 1 & . & . & . & . & . \\[1mm] \hline
                2  & 0 & 1 & 1 & . & . & . & . \\[1mm] \hline
                3  & 0 & 1 & 3 & 1 & . & . & . \\[1mm] \hline
                4  & 0 & 1 & $q+6$ & 6 & 1 & . & . \\[1mm] \hline
                5  & 0 & 1 & $q^2 + 4q + 10$ & $5q+20$ & 10 & 1 & . \\[1mm] \hline
                6  & 0 & 1 & $q^3 + 5q^2 + 10q + 15$ & $q^3 + 9q^2 + 30q + 50$ & $15q + 50$ & 15 & 1 \\[1mm] \hline
                7  & 0 & 1 & $q^4 + 6q^3 + 15q^2 + 20q + 21$ & $q^5 + 5q^4 + 22q^3 + 63q^2 + 105q + 105$ & $7q^3 + 42q^2 + 126q + 175$ & $35q+105$ & 21 \\[1mm] \hline
                8  & 0 & 1 & $q^6 + 8q^5 + 28q^4 + $ & $q^9 + 7q^8 + 29q^7 + 83q^6 + 191q^5 +$
                   & $q^9 + 6q^8 + 30q^7 + 110q^6 + 315q^5 +$ & $9q^6 + 72q^5 + 270q^4 + 804q^3 +$
                   & $84q^3 + 378q^2 $ \\
                   &   &   & $56q^3 + 70q^2 + 56q + 36$ & $ 376q^4 + 616q^3 + 756q^2 + 630q + 336$ & $ 720q^4 + 1380q^3 + 2016q^2 + 2016q + 1176$ 
                   & $ 1680q^2 + 2352q + 1764$ & 
\end{tabular} 
\end{sideways}
\caption{Small values of the $q$-Stirling numbers.}
\end{table}

\setlength{\parindent}{0mm}

\end{document}